\documentclass[secnum,leqno]{rims-bessatsu}
\usepackage{amssymb, amsmath}
\usepackage{graphics}
\usepackage[dvips]{graphicx}
\usepackage{eepic}
\usepackage{epic}
\newtheorem{thm}{Theorem}[section]

\newtheorem{lmm}[thm]{Lemma}

\newtheorem{prp}[thm]{Proposition}

\theoremstyle{definition}

\newtheorem{exa}[thm]{Example}
\theoremstyle{remark}
\newtheorem*{rem}{Remark}
\def\fl#1{\left\lfloor#1\right\rfloor}

\def\sts#1#2{\left\{#1\atop#2\right\}} 
\title{On poly-Euler numbers of the second kind}
\author{Takao \textsc{Komatsu}\footnote{Department of Mathematical Sciences, School of Science, Zhejiang Sci-Tech University, Hangzhou 310018 China. \endgraf e-mail: \texttt{komatsu@zstu.edu.cn}}}
\AuthorHead{Takao Komatsu}         
\VolumeNo{77}           
\YearNo{2020}           
\PagesNo{143--158}      
\communication{Received January 18, 2017.
Revised March 2, 2018.}      
\begin{document}
%

\maketitle

\begin{abstract}      
For an integer $k$, define poly-Euler numbers of the second kind $\widehat E_n^{(k)}$ ($n=0,1,\dots$) by 
$$   
\frac{{\rm Li}_k(1-e^{-4 t})}{4\sinh t}=\sum_{n=0}^\infty\widehat E_n^{(k)}\frac{t^n}{n!}\,.  
$$    
When $k=1$, $\widehat E_n=\widehat E_n^{(1)}$ are {\it Euler numbers of the second kind} or {\it complimentary Euler numbers} defined by 
$$ 
\frac{t}{\sinh t}=\sum_{n=0}^\infty\widehat E_n\frac{t^n}{n!}\,. 
$$ 
Euler numbers of the second kind were introduced as special cases of hypergeometric Euler numbers of the second kind in \cite{KZ}, so that they would supplement hypergeometric Euler numbers.  
In this paper, we give several properties of Euler numbers of the second kind. In particular, we determine their denominators.  We also show several properties of poly-Euler numbers of the second kind, including duality formulae and congruence relations.   
\end{abstract}

\section{Introduction}   

For an integer $k$, poly-Euler numbers $E_n^{(k)}$ ($n=0,1,\dots$) are defined by 
\begin{equation} 
\frac{{\rm Li}_k(1-e^{-4 t})}{4 t\cosh t}=\sum_{n=0}^\infty E_n^{(k)}\frac{t^n}{n!} 
\label{def:peuler}
\end{equation}   
(\cite{OS1,OS2,OS3}), 
where 
$$
{\rm Li}_k(z)=\sum_{n=1}^\infty\frac{z^n}{n^k}\quad(|z|<1,~k\in\mathbb Z)
$$ 
is the $k$-th polylogarithm function.  
When $k=1$, $E_n=E_n^{(1)}$ are the Euler numbers defined by 
\begin{equation}  
\frac{1}{\cosh t}=\sum_{n=0}^\infty E_n\frac{t^n}{n!}\,.
\label{def:euler}
\end{equation}    
Euler numbers have been extensively studied by many authors (see e.g. \cite{KP,OS1,OS2,OS3,Sasaki} and references therein), in particular, by means of Bernoulli numbers.  
In \cite{KZ}, for $N\ge 0$ {\it hypergeometric Euler numbers} $E_{N,n}$ ($n=0,1,2,\dots$) are defined by 
\begin{align} 
\frac{1}{{}_1 F_2(1;N+1,(2 N+1)/2;t^2/4)}&=\frac{t^{2 N}/(2 N)!}{\cosh t-\sum_{n=0}^{N-1}t^{2 n}/(2 n)!}\notag\\
&=\sum_{n=0}^\infty E_{N,n}\frac{t^n}{n!}\,, 
\label{def1:hypergeuler}
\end{align}  
where ${}_1 F_2(a;b,c;z)$ is the hypergeometric function defined by 
$$
{}_1 F_2(a;b,c;z)=\sum_{n=0}^\infty\frac{(a)^{(n)}}{(b)^{(n)}(c)^{(n)}}\frac{z^n}{n!}\,.
$$
Here $(x)^{(n)}$ is the rising factorial, defined by $(x)^{(n)}=x(x+1)\cdots(x+n-1)$ ($n\ge 1$) with $(x)^{(0)}=1$.    
Note that  
When $N=0$,  
$E_n=E_{0,n}$ are the Euler numbers defined in (\ref{def:euler}).  

The sums of products of hypergeometric Euler numbers can be expressed as 
for $N\ge 1$ and $n\ge 0$, 
$$  
\sum_{i=0}^n\binom{n}{i}E_{N,i}E_{N,n-i}
=\sum_{k=0}^n\binom{n}{k}\frac{2 N-k}{2 N}E_{N,k}\widehat E_{N-1,n-k}\,. 
$$  
where $\widehat E_{N,n}$ are the {\it hypergeometric Euler numbers of the second kind} or {\it complementary hypergeometric Euler numbers}  defined by 
\begin{align}  
\frac{1}{{}_1 F_2(1;N+1,(2 N+3)/2;t^2/4)}&=
\frac{t^{2N+1}/(2 N+1)!}{\sinh t-\sum_{n=0}^{N-1}t^{2 n+1}/(2 n+1)!}\notag\\
&=\sum_{n=0}^\infty\widehat E_{N,n}\frac{t^n}{n!} 
\label{def:heuler}
\end{align}   
(\cite[Theorem 4]{KZ}).  
When $n=0$,  
$\widehat E_n=\widehat E_{0,n}$ are the {\it Euler numbers of the second kind} or {\it complementary Euler numbers} defined by  
\begin{equation}  
\frac{t}{\sinh t}=\sum_{n=0}^\infty\widehat E_n\frac{t^n}{n!}\,.   
\label{def:euler2}
\end{equation}   
In \cite{KP}, $\widehat E_n$ are called {\it weighted Bernoulli numbers}.  But they mean different in different literatures. 
On the other hand,  the sums of products of hypergeometric Euler numbers of the second kind can be also expressed as 
$$  
\sum_{i=0}^n\binom{n}{i}\widehat E_{N,i}\widehat E_{N,n-i}
=\sum_{k=0}^n\binom{n}{k}\frac{2 N-k+1}{2 N+1}\widehat E_{N,k}E_{N,n-k}  
$$  
(\cite[Theorem 6]{KZ}).

Euler numbers of the second kind are complementary in view of determinants too.   
It is known that the Euler numbers are given by the determinant 
\begin{equation}  
E_{2n}=(-1)^n (2n)!
\begin{vmatrix}   
\frac{1}{2!}& 1 &~& ~&~\\
\frac{1}{4!}&  \frac{1}{2!} & 1 &~&~\\
\vdots & ~  &  \ddots~~ &\ddots~~ & ~\\
\frac{1}{(2n-2)!}& \frac{1}{(2n-4)!}& ~&\frac{1}{2!} &  1\\
\frac{1}{(2n)!}&\frac{1}{(2n-2)!}& \cdots &  \frac{1}{4!} & \frac{1}{2!}
\end{vmatrix}
\label{euler:det}
\end{equation}   
(\textit{Cf.} \cite[p.52]{Glaisher}).  
Euler numbers of the second kind (\cite[Corollary 2.2]{Ko2}) can be expressed as 
\begin{equation}  
\widehat E_{2 n}=(-1)^n(2 n)!
\begin{vmatrix}   
\frac{1}{3!}& 1 &~& ~&~\\
\frac{1}{5!}&  \frac{1}{3!} & 1 &~&~\\
\vdots & ~  &  \ddots~~ &\ddots~~ & ~\\
\frac{1}{(2n-1)!}& \frac{1}{(2n-3)!}& ~&\frac{1}{3!} &  1\\
\frac{1}{(2n+1)!}&\frac{1}{(2n-1)!}& \cdots &  \frac{1}{5!} & \frac{1}{3!}
\end{vmatrix}\,. 
\label{euler2:det}
\end{equation} 
Since Bernoulli numbers can be expressed as 
$$ 
B_n= 
(-1)^n n! 
\begin{vmatrix}   
\frac{1}{2!}& 1 &~& ~&~\\
\frac{1}{3!}&  \frac{1}{2!} & 1 &~&~\\
\vdots & ~  &  \ddots~~ &\ddots~~ & ~\\
\frac{1}{(n-1)!}& \frac{1}{(n-2)!}& ~&\frac{1}{2!} &  1\\
\frac{1}{n!}&\frac{1}{(n-1)!}& \cdots &  \frac{1}{3!} & \frac{1}{2!}
\end{vmatrix}
$$ 
(\textit{Cf.} \cite[p.53]{Glaisher}),  Euler numbers and those of second kind fill the gaps each other in Bernoulli numbers.

In \cite[Proposition 1.1]{KZ}, it is shown that hypergeometric Euler numbers $E_{N,n}$ satisfy the relation: 
$$ 
\sum_{i=0}^{n/2}\frac{1}{(2 N+n-2 i)!(2 i)!}E_{N,2 i}=0\quad\hbox{($n\ge 2$ is even)}
$$ 
with $E_{N,0}=1$.  

From (\ref{def:heuler}), we have 
\begin{align*} 
\frac{t^{2 N+1}}{(2 N+1)!}&=\left(\sum_{n=N}^\infty\frac{t^{2 n+1}}{(2 n+1)!}\right)\left(\sum_{n=0}^\infty\widehat E_{N,n}\frac{t^n}{n!}\right)\\
&=t^{2 N+1}\left(\sum_{n=0}^\infty\frac{\frac{1+(-1)^n}{2}t^n}{(2 N+n+1)!}\right)\left(\sum_{n=0}^\infty\widehat E_{N,n}\frac{t^n}{n!}\right)\\
&=t^{2 N+1}\sum_{n=0}^\infty\left(\sum_{i=0}^n\frac{\frac{1+(-1)^{n-i}}{2}}{(2 N+n-i+1)!}\frac{\widehat E_{N,i}}{i!}\right)t^n\,. 
\end{align*}  
Therefore, 
the hypergeometric Euler numbers of the second kind satisfy the recurrence relation for even $n\ge 2$  
$$
\sum_{i=0}^{n/2}\frac{\widehat E_{N,2 i}}{(2 N+n-2 i+1)!(2 i)!}=0
$$   
or for $n\ge 1$ 
\begin{equation} 
\widehat E_{N, 2 n}=-(2 n)!(2 N+1)!\sum_{i=0}^{n-1}\frac{\widehat E_{N,2 i}}{(2 N+2 n-2 i+1)!(2 i)!}\,. 
\label{che-rel}
\end{equation}  

It turns that $\widehat E_{N, 2 n}$ can be given by the determinant (\cite[Theorem 2.1]{Ko2}).  

\begin{thm}  
For $N\ge 0$ and $n\ge 1$, we have 
$$ 
\widehat E_{N, 2 n}=(-1)^n(2 n)!
\left|
\begin{array}{cccc}
\frac{(2 N+1)!}{(2 N+3)!}&1&&\\
\frac{(2 N+1)!}{(2 N+5)!}&\ddots&\ddots&\\
\vdots&&\ddots&1\\
\frac{(2 N+1)!}{(2 N+2 n+1)!}&\cdots&\frac{(2 N+1)!}{(2 N+5)!}&\frac{(2 N+1)!}{(2 N+3)!}
\end{array} 
\right|\,. 
$$ 
\label{th:hge2det} 
\end{thm}    

\noindent 
When $N=0$, we obtain the determinant expression of Euler numbers of the second kind in (\ref{euler2:det}).

Similarly to Theorem \ref{th:hge2det}, we get the determinant expression of hypergeometric Euler numbers (\cite[Theorem 2.3]{Ko2}).  

\begin{thm}  
For $N\ge 0$ and $n\ge 1$, we have 
$$ 
E_{N, 2 n}=(-1)^n(2 n)!
\left|
\begin{array}{cccc}
\frac{(2 N)!}{(2 N+2)!}&1&&\\
\frac{(2 N)!}{(2 N+4)!}&\ddots&\ddots&\\
\vdots&&\ddots&1\\
\frac{(2 N)!}{(2 N+2 n)!}&\cdots&\frac{(2 N)!}{(2 N+4)!}&\frac{(2 N)!}{(2 N+2)!}
\end{array} 
\right|\,. 
$$ 
\end{thm} 

\noindent 
When $N=0$, we obtain the determinant expression of Euler numbers in (\ref{euler:det}).

In Section 2, we shall show several properties of Euler numbers of the second kind. In particular, we determine the denominator of $\widehat E_{2 n}$. 
In Section 3, we introduce poly-Euler numbers of the second kind as one directed generalizations of the original Euler numbers of the second kind. We give some expressions of poly-Euler numbers of the second kind with both positive and negative indices.  
In Section 4, we show one type of duality formula for poly-Euler numbers of the second kind. 
In Section 5, we shall give several congruence relations of poly-Euler numbers of the second kind with negative indices.

\section{Euler numbers of the second kind}   

In this section, we shall show several properties of Euler numbers of the second kind. In particular, we determine the denominator of $\widehat E_{2 n}$. We also give some identities involving Euler numbers of the second kind, as analogous results of those in Euler numbers.  

From the definitions (\ref{def:euler}) and (\ref{def:euler2}), 
$$
E_{2 n+1}=\widehat E_{2 n+1}=0\quad(n\ge0)\,. 
$$ 
We also know that 
\begin{equation} 
\frac{1}{\cos t}=\sum_{n=0}^\infty(-1)^n E_{2 n}\frac{t^{2 n}}{(2 n)!}
\label{def:euler-cos} 
\end{equation} 
and 
\begin{equation} 
\frac{t}{\sin t}=\sum_{n=0}^\infty(-1)^n\widehat E_{2 n}\frac{t^{2 n}}{(2 n)!}
\label{def:euler2-sin} 
\end{equation} 

Euler numbers $E_{2 n}$ are integers, but Euler numbers of the second kind $\widehat E_{2 n}$ are rational numbers. We can know the denominator of $\widehat E_{2 n}$ completely.  

\begin{thm}  
For an integer $n\ge 1$, the denominator of Euler numbers of the second kind $\widehat E_{2 n}$ is given by 
$$
\prod_{(p-1)|2 n}p\,, 
$$ 
where $p$ runs over all odd primes with $(p-1)|2 n$. 
In other word, 
$$
\left(\prod_{(p-1)|2 n}p\right)\widehat E_{2 n}
$$ 
is an integer, where $p$ runs over all odd primes with $(p-1)|2 n$.  
\label{den:euler2}
\end{thm} 
\begin{proof}  
Notice that for $n\ge 1$, we have 
$$
\widehat E_n=2^n\mathcal B_n\left(\frac{1}{2}\right)=(2-2^n)\mathcal B_n\,, 
$$ 
where $\mathcal B_n(x)$ is the Bernoulli polynomial, defined by 
$$
\frac{t e^{t x}}{e^t-1}=\sum_{n=0}^\infty\mathcal B_n(x)\frac{t^n}{n!}\,. 
$$ 
When $x=0$, $\mathcal B_n=\mathcal B_n(0)$ is the classical Bernoulli number with $\mathcal B_1=-1/2$.  
By Von Staud-Clausen theorem, for $n\ge 1$ 
$$
\mathcal B_{2 n}+\sum_{(p-1)|2 n}\frac{1}{p} 
$$  
is an integer, where the sum extends over all primes $p$ with $(p-1)|2 n$.  
By Fermat's Little Theorem, if $(p-1)|2 n$, then $m^{2 n}\equiv 1\pmod p$ for $m=1,2,\dots,p-1$.  Thus, $2^{2 n}\equiv 1\not\equiv 2\pmod p$ for any odd prime $p$.  Therefore, the denominator of Euler numbers of the second kind is given by 
$$
\prod_{(p-1)|2 n}p
$$ 
where the product extends over all odd primes $p$ with $(p-1)|2 n$. 
\end{proof}  
\begin{exa} 
The odd primes $p$ satisfying $(p-1)|24$ are $3,5,7,13$, and 
$$
\widehat E_{24}=\frac{1982765468311237}{1365}=\frac{47\cdot 103\cdot 178481\cdot 2294797}{3\cdot 5\cdot 7\cdot 13}\,.  
$$ 
The odd prime $p$ satisfying $(p-1)|26$ is $3$, and 
$$
\widehat E_{26}=-\frac{286994504449393}{3}=-\frac{13\cdot 31\cdot 601\cdot 1801\cdot 657931}{3}\,.  
$$ 
\end{exa} 
\begin{rem} 
For any integer $n\ge 0$, 
$$
(2 n+1)(2 n-1)\cdots 3\widehat E_{2 n}=\frac{(2 n+1)!}{2^n n!}\widehat E_{2 n}
$$ 
is an integer.  
\end{rem}

It is known that Euler numbers satisfy the recurrence relation 
$$
\sum_{j=0}^n\binom{2 n}{2 j}E_{2 j}=0\quad(n\ge 1)
$$  
with $E_0=1$.  Similarly, Euler numbers of the second kind satisfy the following recurrence relation.  

\begin{thm}  
For $n\ge 1$, 
$$
\sum_{j=0}^n\binom{2 n+1}{2 j}\widehat E_{2 j}=0
$$ 
and $\widehat E_0=1$.  
\label{e2-rec} 
\end{thm} 
\begin{proof} 
From the definition (\ref{def:euler2}), we have 
\begin{align*}  
t&=\frac{t}{\sinh t}\sinh t\\
&=\left(\sum_{j=0}^\infty\widehat E_{2 j}\frac{t^{2 j}}{(2 j)!}\right)\left(\sum_{l=0}^\infty\frac{t^{2 l+1}}{(2 l+1)!}\right)\\
&=\sum_{n=0}^\infty\sum_{j=0}^n\binom{2 n+1}{2 j}\widehat E_{2 j}\frac{t^{2 n+1}}{(2 n+1)!}\quad(n=j+l)\,.
\end{align*} 
Comparing the coefficients on both sides, we get the result. 
\end{proof}

For a positive integer $n$ and a nonnegative integer $k$, Euler numbers satisfy the relation 
$$
\sum_{j=0}^n\binom{2 n}{2 j}(2 k+1)^{2 n-2 j}E_{2 j}=2\sum_{l=1}^k(-1)^{k-l}(2 l)^{2 n} 
$$ 
(e.g. \cite{Liu}).  Euler numbers of the second kind satisfy the following relation.  

\begin{thm}  
For a positive integer $n$ and a nonnegative integer $k$, 
$$
\sum_{j=0}^n\binom{2 n+1}{2 j}(2 k+1)^{2 n-2 j+1}\widehat E_{2 j}=2(2 n+1)\sum_{l=1}^k(2 l)^{2 n}\,.   
$$ 
\label{sum1}
\end{thm}  
\begin{proof}  
Put 
$$
A(t)=\sum_{k=0}^\infty t^k\cos kx \quad\hbox{and}\quad 
B(t)=\sum_{k=0}^\infty t^k\sin kx\,. 
$$ 
For $|t|<1$, we have 
\begin{align*} 
A(t)+\sqrt{-1}B(t)&=\sum_{k=0}^\infty t^k(\cos x+\sqrt{-1}\sin x)^k\\
&=\frac{1}{1-t\cos x-\sqrt{-1}t\sin x}
=\frac{1-t\cos x+\sqrt{-1}t\sin x}{1-2 t\cos x+t^2}\,.
\end{align*} 
Hence, we get 
$$
A(t)=\frac{1-t\cos x}{1-2 t\cos x+t^2}\quad\hbox{and}\quad 
B(t)=\frac{t\sin x}{1-2 t\cos x+t^2}\quad(|t|<1)\,,  
$$ 
yielding  
\begin{align*} 
\sum_{k=0}^\infty t^k\cos(2 k+1)x=\frac{A(\sqrt{t})-A(-\sqrt{t})}{2\sqrt{t}}=\frac{(1-t)\cos x}{(1+t)^2-4 t\cos^2 x}\,,\\
\sum_{k=0}^\infty t^k\cos 2 k x=\frac{A(\sqrt{t})+A(-\sqrt{t})}{2}=\frac{1-2 t\cos^2 x+t}{(1+t)^2-4 t\cos^2 x}\,,\\
\sum_{k=0}^\infty t^k\sin(2 k+1)x=\frac{B(\sqrt{t})-B(-\sqrt{t})}{2\sqrt{t}}=\frac{(1+t)\sin  x}{(1+t)^2-4 t\cos^2 x}\,,\\
\sum_{k=0}^\infty t^k\sin 2 k x=\frac{B(\sqrt{t})+B(-\sqrt{t})}{2}=\frac{t\sin 2 x}{(1+t)^2-4 t\cos^2 x}\,. 
\end{align*} 
Thus, for $|t|<1$, we have 
\begin{align*} 
\sum_{k=0}^\infty t^k\sum_{j=0}^{2 k}\cos(2 k-2 j)x&=\sum_{k=0}^\infty t^k\left(2\sum_{l=0}^\infty t^l\cos 2 l x-1\right)\\
&=2\left(\sum_{k=0}^\infty t^k\right)\left(\sum_{l=0}^\infty t^l\cos 2 l x\right)-\sum_{k=0}^\infty t^k\\
&=\frac{2}{1-t}\frac{1-2 t\cos^2 x+t}{(1+t)^2-4 t\cos^2 x}-\frac{1}{1-t}\\
&=\frac{1+t}{(1+t)^2-4 t\cos^2 x}=\frac{1}{\sin x}\sum_{k=0}^\infty t^k\sin(2 k+1)x\,. 
\end{align*} 
Therefore, we obtain 
\begin{equation}  
\sum_{j=0}^{2 k}\cos(2 k-2 j)x=\frac{\sin(2 k+1)x}{\sin x}\,. 
\label{cos-sin/sin}
\end{equation}  
The right-hand side of (\ref{cos-sin/sin}) is equal to 
\begin{align*}
&\frac{x}{\sin x}\frac{\sin(2 k+1)x}{x}\\
&=\left(\sum_{j=0}^\infty(-1)^j\widehat E_{2 j}\frac{x^{2 j}}{(2 j)!}\right)\left(\sum_{m=0}^\infty(-1)^m(2 k+1)^{2 m+1}\frac{x^{2 m}}{(2 m+1)!}\right)\\
&=\sum_{n=0}^\infty(-1)^n\sum_{j=0}^n\binom{2 n+1}{2 j}(2 k+1)^{2 n-2 j+1}\widehat E_{2 j}\frac{x^{2 n}}{(2 n+1)!}\quad(j+m=n)\,.
\end{align*} 
The left-hand side of (\ref{cos-sin/sin}) is equal to 
$$
\sum_{j=0}^{2 k}\sum_{n=0}^\infty(-1)^n(2 k-2 j)^{2 n}\frac{x^{2 n}}{(2 n)!}\,. 
$$ 
Comparing the coefficients on both sides, we have 
\begin{align*} 
\sum_{j=0}^n\binom{2 n+1}{2 j}(2 k+1)^{2 n-2 j+1}\frac{\widehat E_{2 j}}{(2 n+1)!}&=\sum_{l=0}^{2 k}\frac{(2 k-2 l)^{2 n}}{(2 n)!}\\
&=2\sum_{l=1}^k\frac{(2 l)^{2 n}}{(2 n)!}\,. 
\end{align*} 
Therefore, we get the desired result.  
\end{proof}

\section{Poly-Euler numbers of the second kind}  

In this section we introduce poly-Euler numbers of the second kind as one directed generalizations of the original Euler numbers of the second kind.  A different direction of generalizations is in (\ref{def:heuler}) as hypergeometric Euler numbers of the second kind. Similar poly numbers are poly-Bernoulli numbers (\cite{Kaneko}) and poly-Cauchy numbers (\cite{Ko1}).  We shall give some expressions of poly-Euler numbers of the second kind with both positive and negative indices.

For an integer $k$, define {\it poly-Euler numbers of the second kind} $\widehat E_n^{(k)}$ ($n=0,1,\dots$) by 
\begin{equation}  
\frac{{\rm Li}_k(1-e^{-4 t})}{4\sinh t}=\sum_{n=0}^\infty\widehat E_n^{(k)}\frac{t^n}{n!}\,.  
\label{def:peuler2} 
\end{equation}   
When $k=1$, $\widehat E_n=\widehat E_n^{(1)}$ are {\it Euler numbers of the second kind} or complimentary Euler numbers defined in (\ref{def:euler2}).   
Several values of poly-Euler numbers of the second kind can be seen in Table 1.


\begin{table}[phtb]  
  \begin{center}
    \caption{The numbers $\widehat E_n^{(k)}$ for $1\le n\le 7$ and $1\le k\le 5$} 
    \begin{tabular}{|c|ccccc|} \hline
    $k$&$1$&$2$&$3$&$4$&$5$ \\ \hline 
    $\widehat E_{1}^{(k)}$&$0$&$-1$&$-\frac{3}{2}$&$-\frac{7}{4}$&$-\frac{15}{8}$ \\ 
    $\widehat E_{2}^{(k)}$&$-\frac{1}{3}$&$\frac{5}{9}$&$\frac{59}{27}$&$\frac{275}{81}$&$\frac{1004}{243}$ \\   
    $\widehat E_{3}^{(k)}$&$0$&$1$&$-\frac{11}{6}$&$-\frac{211}{36}$&$-\frac{985}{108}$\\  
    $\widehat E_{4}^{(k)}$&$\frac{7}{15}$&$-\frac{679}{225}$&$-\frac{12737}{3375}$&$\frac{245789}{50625}$&$\frac{12383617}{759375}$\\  
    $\widehat E_{5}^{(k)}$&$0$&$-\frac{7}{3}$&$\frac{527}{30}$&$\frac{47171}{2700}$&$-\frac{85361}{9000}$\\  
    $\widehat E_{6}^{(k)}$&$-\frac{31}{21}$&$\frac{60001}{2205}$&$\frac{483221}{231525}$&$-\frac{1961354909}{24310125}$&$-\frac{205924986214}{2552563125}$\\  
    $\widehat E_{7}^{(k)}$&$0$&$\frac{31}{3}$&$-\frac{45853}{210}$&$-\frac{1250393}{132300}$&$\frac{763114237}{2315250}$\\ \hline
    \end{tabular}
  \end{center}
\end{table}


Poly-Euler numbers of the second kind can be expressed explicitly in terms of poly-Bernoulli numbers $B_n^{(k)}$ (\cite{Kaneko}) defined by 
$$
\frac{{\rm Li}_k(1-e^{-t})}{1-e^{-t}}=\sum_{n=0}^\infty B_n^{(k)}\frac{t^n}{n!}\,.
$$ 
When $k=1$, $B_n=B_n^{(1)}$ is the Bernoulli number with $B_1=1/2$.  
Notice that poly-Bernoulli numbers can be expressed explicitly (\cite[Theorem 1]{Kaneko}) in terms of the Stirling numbers of the second kind $\sts{m}{j}$:  
$$
B_m^{(k)}=\sum_{j=0}^m\frac{(-1)^{m-j}j!}{(j+1)^k}\sts{m}{j}\,. 
$$ 
Here, the Stirling numbers of the second kind are defined by 
$$
\sts{n}{k}=\frac{1}{k!}\sum_{j=0}^k(-1)^{k-j}\binom{k}{j}j^n\,,  
$$ 
yielding from  
$$
x^n=\sum_{k=0}^n\sts{n}{k}x(x-1)\cdots(x-k+1)\,. 
$$

\begin{lmm}  
For integers $n$ and $k$ with $n\ge 0$, we have 
$$
\widehat E_n^{(k)}=\frac{1}{2}\sum_{m=0}^n\binom{n}{m}4^m\bigl((-1)^{n-m}+(-3)^{n-m}\bigr)B_m^{(k)}\,. 
$$ 
\label{peuler-pbernoulli}
\end{lmm}

When the index is negative, we had a more explicit formula without Bernoulli numbers \cite{Ko2}.

\begin{lmm} 
For nonnegative integers $n$ and $k$, we have 
$$
\widehat E_n^{(-k)}=\frac{(-1)^k}{2}\sum_{l=0}^k(-1)^l l!\sts{k}{l}\bigl((4 l+3)^n+(4 l+1)^n\bigr)\,. 
$$ 
\label{peuler-negative}
\end{lmm}  

Lemma \ref{peuler-negative} can be stated as follows too.  

\begin{lmm} 
For nonnegative integers $n$ and $k$, we have 
$$
\widehat E_n^{(-k)}=(-1)^k\sum_{l=0}^k(-1)^l l!\sts{k}{l}\sum_{m=0}^{\fl{\frac{n}{2}}}\binom{n}{2 m}(4 l+2)^{n-2 m}\,. 
$$ 
\label{peuler-negative-2}
\end{lmm}


\begin{table}[phtb]  
  \begin{center}
    \caption{The numbers $\widehat E_n^{(-k)}$ for $1\le n\le 7$ and $0\ge -k\ge -4$} \medskip 
    \begin{tabular}{|c|ccccc|} \hline
    $-k$&$0$&$-1$&$-2$&$-3$&$-4$ \\ \hline   
    $\widehat E_{1}^{(-k)}$&$2$&$6$&$14$&$30$&$62$ \\  
    $\widehat E_{2}^{(-k)}$&$5$&$37$&$165$&$613$&$2085$ \\   
    $\widehat E_{3}^{(-k)}$&$14$&$234$&$1826$&$10770$&$55154$\\  
    $\widehat E_{4}^{(-k)}$&$41$&$1513$&$19689$&$175465$&$1287657$\\   
    $\widehat E_{5}^{(-k)}$&$122$&$9966$&$210134$&$2741670$&$27930182$\\  
    $\widehat E_{6}^{(-k)}$&$365$&$66637$&$2236365$&$41809933$&$578341965$\\ 
    $\widehat E_{7}^{(-k)}$&$1094$&$450834$&$23819306$&$628464090$&$11615023034$\\ \hline
    \end{tabular}
  \end{center}
\end{table}

Several exact values can be seen in Table 2.  
As special cases, we have the following.  

\begin{lmm} 
For nonnegative integers $n$ and $k$, we have 
\begin{align*} 
&\widehat E_0^{(-k)}=1,\quad \widehat E_1^{(-k)}=2^{k+2}-2,\quad 
\widehat E_2^{(-k)}=32\cdot 3^k-2^{k+5}+5,\\ 
&\widehat E_3^{(-k)}=384\cdot 4^k-576\cdot 3^k+220\cdot 2^k-14,\\
&\widehat E_4^{(-k)}=6144\cdot 5^k-12288\cdot 4^k+7616\cdot 3^k-1472\cdot 2^k+41,\\
&\widehat E_5^{(-k)}=122880\cdot 6^k-307200\cdot 5^k+264960\cdot 4^k-90240\cdot 3^k+9844\cdot 2^k-122,\\ 
&\widehat E_n^{(0)}=\frac{3^n+1}{2},\quad \widehat E_n^{(-1)}=\frac{7^n+5^n}{2},\quad \widehat E_n^{(-2)}=\frac{2(11^n+9^n)-(7^n+5^n)}{2},\\
&\widehat E_n^{(-3)}=\frac{6(15^n+13^n)-6(11^n+9^n)+(7^n+5^n)}{2},\\ 
&\widehat E_n^{(-4)}=\frac{24(19^n+17^n)-36(15^n+13^n)+14(11^n+9^n)-(7^n+5^n)}{2},\\
&\widehat E_n^{(-5)}=\frac{120(23^n+21^n)-240(19^n+17^n)+150(15^n+13^n)-30(11^n+9^n)+(7^n+5^n)}{2}\,.
\end{align*} 
\label{special-peuler-neg}
\end{lmm}

\section{Duality formulae for poly-Euler numbers of the second kind} 

It is known that the duality formula $B_n^{(-k)}=B_k^{(-n)}$ ($n,k\ge 0$) holds for poly-Bernoulli numbers (\cite{Kaneko}). In this section, we shall show a different type of duality formula for poly-Euler numbers of the second kind.  

\begin{thm}  
For nonnegative integers $n$ and $k$, we have  
$$
\sum_{m=0}^n\binom{n}{m}\frac{2-E_{n-m}}{4^n}\widehat E_m^{(-k)}
=\sum_{m=0}^k\binom{k}{m}\frac{2-E_{k-m}}{4^k}\widehat E_m^{(-n)}\,.
$$ 
\label{th_duality1} 
\end{thm}  
 
This theorem is proven by using the expression of poly-Bernoulli numbers in terms of poly-Euler numbers of the second kind.  In \cite[Theorem 3]{Ko2}, poly-Euler numbers of the second kind are expressed in terms of poly-Bernoulli numbers:  
$$
\widehat E_n^{(k)}=\frac{1}{2}\sum_{m=0}^n\binom{n}{m}4^m\bigl((-1)^{n-m}+(-3)^{n-m}\bigr)B_m^{(k)}\,.
$$ 
 
\begin{prp}  
For integers $n$ and $k$ with $n\ge 0$, we have 
$$
B_n^{(k)}=\sum_{m=0}^n\binom{n}{m}\frac{2-E_{n-m}}{4^n}\widehat E_m^{(k)}\,. 
$$ 
\label{prp_pb-pe2} 
\end{prp}   
\begin{proof}  
Since 
\begin{equation} 
\frac{{\rm Li}_k(1-e^{-4 t})}{1-e^{-4 t}}=\frac{2}{e^{-t}+e^{-3 t}}\frac{{\rm Li}_k(1-e^{-4 t})}{2(e^t-e^{-t})}\,, 
\label{eq:124}
\end{equation}  
we have 
\begin{align*} 
\sum_{n=0}^\infty B_n^{(k)}\frac{(4 t)^n}{n!}&=\left(2 e^t-\frac{1}{\cosh t}\right)\left(\sum_{n=0}^\infty\widehat E_n^{(k)}\frac{t^n}{n!}\right)\\
&=\left(2\sum_{l=0}^\infty\frac{t^l}{l!}\right)\left(\sum_{m=0}^\infty\widehat E_m^{(k)}\frac{t^m}{m!}\right)-\left(\sum_{l=0}^\infty E_l\frac{t^l}{l!}\right)\left(\sum_{m=0}^\infty\widehat E_m^{(k)}\frac{t^m}{m!}\right)\\ 
&=2\sum_{n=0}^\infty\sum_{m=0}^n\binom{n}{m}\widehat E_m^{(k)}\frac{t^n}{n!}-\sum_{n=0}^\infty\sum_{m=0}^n\binom{n}{m}E_{n-m}\widehat E_m^{(k)}\frac{t^n}{n!}\,.
\end{align*} 
Comparing the coefficients on both sides, we get 
$$
4^n B_n^{(k)}=\sum_{m=0}^n\binom{n}{m}(2-E_{n-m})\widehat E_m^{(k)}\,.
$$ 
\end{proof}

\begin{proof}[Proof of Theorem \ref{th_duality1}]   
From Proposition \ref{prp_pb-pe2} and the duality formula $B_n^{(-k)}=B_k^{(-n)}$, we get the desired result. 
\end{proof}  
\bigskip

We can also describe the positivity of poly-Euler numbers of the second kind with negative index.  

\begin{thm} 
For nonnegative integers $n$ and $k$, we have 
\begin{align*} 
\widehat E_n^{(-k)}&=\sum_{j=0}^{\min(n,k)}(j!)^2\sum_{m=0}^n\sum_{\mu=0}^k\binom{n}{m}\binom{k}{\mu}\sts{n-m}{j}\sts{\mu}{j}4^{n-m}\widehat E_m^{(0)}\\
&=\sum_{j=0}^{\min(n,k)}(j!)^2\sum_{m=0}^n\sum_{\mu=0}^k\binom{n}{m}\binom{k}{\mu}\sts{n-m}{j}\sts{\mu}{j}\frac{4^{n-m}(3^m+1)}{2}\,. 
\end{align*} 
\label{positive} 
\end{thm} 
\begin{proof}  
From (\ref{eq:124}), we have 
\begin{align*}  
\sum_{k=0}^\infty\frac{{\rm Li}_{-k}(1-e^{-4 x})}{4\sinh x}\frac{(4 y)^k}{k!}&=\frac{e^{-x}+e^{-3 x}}{2}\sum_{k=0}^\infty\frac{{\rm Li}_{-k}(1-e^{-4 x})}{1-e^{-4 x}}\frac{(4 y)^k}{k!}\\
&=e^{-4 x}\frac{{\rm Li}_0(1-e^{-4 x})}{4\sinh x}\frac{e^{4(x+y)}}{e^{4 x}+e^{4 y}-e^{4(x+y)}}\\
&=\sum_{m=0}^\infty\widehat E_m^{(0)}\frac{x^m}{m!}e^{4 y}\sum_{j=0}^\infty(j!)^2\frac{(e^{4 x}-1)^j}{j!}\frac{(e^{4 y}-1)^j}{j!}\\
&=\sum_{j=0}^\infty(j!)^2\left(\sum_{m=0}^\infty\widehat E_m^{(0)}\frac{x^m}{m!}\right)\left(\sum_{\nu=j}^\infty\sts{\nu}{j}\frac{(4 x)^\nu}{\nu!}\right)\\
&\qquad\times\left(\sum_{l=0}^\infty\frac{(4 y)^l}{l!}\right)\left(\sum_{\mu=j}^\infty\sts{\mu}{j}\frac{(4 y)^\nu}{\nu!}\right)\\
&=\sum_{j=0}^\infty(j!)^2\sum_{n=0}^\infty\sum_{m=0}^n\binom{n}{m}\widehat E_m^{(0)}\sts{n-m}{j}4^{n-m}\frac{x^n}{n!}\\
&\qquad\times\sum_{k=0}^\infty\sum_{\mu=0}^k\binom{k}{\mu}\sts{\mu}{j}4^k\frac{y^k}{k!}\,. 
\end{align*}  
Since the left hand side is equal to 
$$
\sum_{k=0}^\infty\sum_{n=0}^\infty\widehat E_n^{(-k)}\frac{x^n}{n!}\frac{(4 y)^k}{k!}\,, 
$$
comparing the coefficients on both sides, we get the desired result.  
\end{proof}

\begin{rem}  
When $k=1$ in Theorem \ref{positive}, we have 
$$
\widehat E_n^{(-1)}=
\begin{cases}  
0&\text{if $n=0$};\\
\sum_{m=0}^n\binom{n}{m}\sts{n-m}{1}4^{n-m}\widehat E_m^{(0)}&\text{if $n\ge 1$}\,. 
\end{cases} 
$$ 
It matches the result in Lemma \ref{special-peuler-neg}, as 
$$
\widehat E_n^{(-1)}=\frac{7^n+5^n}{2}(n\ge 0)\,.
$$  
\end{rem}

\section{Congruence relations} 

Poly-Euler numbers of the second kind with positive indices are rational numbers, but those with negative indices are integers. Hence, it is worthwhile considering congruence relations.  

In \cite{Ko2}, we determined the parity and the divisibility of poly-Euler numbers of the second kind as follows.

\begin{lmm}  
For any nonnegative integer $k$, we have 
$$
\widehat E_n^{(-k)}\equiv 
\begin{cases}  
0\pmod 2&\text{if $n$ is odd},\\
1\pmod 2&\text{if $n$ is even}. 
\end{cases} 
$$ 
\label{peuler-parity} 
\end{lmm}   

\begin{lmm}  
For an odd prime $p$ with $p>3$ and a nonnegative integer $k$, we have 
$$ 
\widehat E_p^{(-k)}\equiv 2^{k+2}-2\pmod p\,.  
$$ 
For a nonnegative integer $k$, we have 
\begin{align*} 
\widehat E_3^{(-k)}&\equiv (-1)^k+1\pmod 3\,,\\
\widehat E_2^{(-k)}&\equiv 0\pmod 2\,. 
\end{align*} 
\label{peuler-pmodulo}
\end{lmm}  

In this section, we shall give some more congruence relations of poly-Euler numbers of the second kind with negative indices.

\begin{prp} 
Let $p$ be an odd prime, and $k$ be a fixed nonnegative integer.  Then for integers $n$ and $m$ with $n,m\ge 0$ and $n\equiv m\pmod{p-1}$, we have 
$$
\widehat E_n^{(-k)}\equiv\widehat E_m^{(-k)}\pmod p\,. 
$$ 
\label{prp_cong1}
\end{prp} 
\begin{proof}  
By Fermat's Little Theorem, if $n\equiv m\pmod{p-1}$, then 
$$
(4 l+3)^n\equiv(4 l+3)^m\quad\hbox{and}\quad (4 l+1)^n\equiv(4 l+1)^m\pmod p\,. 
$$ 
By Lemma \ref{peuler-negative} with the fact that $(4 l+3)^n+(4 l+1)^n$ is even for $n\ge 0$, we get the desired result. 
\end{proof}

\begin{exa}  
Let $p=5$ and $k=3$. As $6\equiv 2\pmod 4$, 
$$
\widehat E_6^{(-3)}-\widehat E_2^{(-3)}=41809933-613=5\cdot 8361864\,. 
$$ 
Let $p=3$ and $k=4$. As $7\equiv 5\pmod 2$, 
$$
\widehat E_7^{(-4)}-\widehat E_5^{(-4)}=11615023034-27930182=3\cdot 3862364284\,. 
$$ 
\end{exa}

\begin{thm}  
Let $p$ be an odd prime. If $k\equiv p-2\pmod{p-1}$ for odd integers $n$ and $k$, then we have 
$$
\widehat E_n^{(-k)}\equiv 0\pmod p\,. 
$$ 
\label{thm-cong2} 
\end{thm} 
\begin{proof}  
Notice that if $l\ge p$ then $l!\equiv 0\pmod p$,  and if $(p-1)\not|k$, $\sts{k}{p-1}\equiv 0\pmod p$.  
Hence, by Lemma \ref{peuler-negative-2}, 
\begin{align*} 
\widehat E_n^{(-k)}&=(-1)^k\sum_{l=0}^k(-1)^l l!\sts{k}{l}B(n,l)\\
&=(-1)^k\sum_{l=0}^{p-2}(-1)^l l!\sts{k}{l}B(n,l)\,,   
\end{align*} 
where  
\begin{align*} 
B(n,l)&:=\sum_{m=0}^{\fl{\frac{n}{2}}}\binom{n}{2 m}(4 l+2)^{n-2 m}\\
&=\begin{cases}  
0\pmod p&\text{if $l=\frac{p-1}{2}$};\\
-B(n,p-l-1)\pmod p&\text{if $l=1,2,\dots,\frac{p-3}{2}$}\,. 
\end{cases} 
\end{align*} 
Thus, we get 
\begin{align*}  
\widehat E_n^{(-k)}&\equiv-\sum_{l=1}^{p-2}(-1)^l l!\sts{p-2}{l}B(n,l)\\
&\equiv-\sum_{l=1}^{(p-3)/2}(-1)^l l!\sts{p-2}{l}\bigl(B(n,l)+B(n,p-l-1)\bigr)\\
&\equiv 0\pmod p\,. 
\end{align*} 
Therefore, we have the desired result.   
\end{proof}

\begin{exa}  
Let $p=5$.  Then for any odd number $n$ we have $\widehat E_n^{(-3)}\equiv 0\pmod 5$.   
Together with Lemma \ref{peuler-parity}, we can know that $\widehat E_n^{(-3)}\equiv 0\pmod{10}$.  As seen in Table 2, all of 
$$
\widehat E_1^{(-3)}=30,\quad \widehat E_3^{(-3)}=10770,\quad \widehat E_5^{(-3)}=2741670\quad\hbox{and}\quad \widehat E_7^{(-3)}=628464090
$$ 
are divided by $5$. 
 
Let $p=7$.  Then for any odd number $n$ we have $E_n^{(-5)}\equiv 0\pmod 7$.  
\end{exa}

\section*{Acknowledgement}
 
The author thanks the referee for careful reading of this manuscript and for many helpful suggestions.


\end{document}